\documentclass[11pt]{article}
\usepackage{graphicx} 
\usepackage{amsmath,amsthm,amssymb}                
\usepackage{graphicx, color}
\usepackage{amssymb}
\usepackage{epstopdf}
\usepackage{pdfsync}
\usepackage{makeidx}

\usepackage{geometry} 
\newtheorem{definition}{Definition}[section]
\newtheorem{lemma}[definition]{Lemma}
\newtheorem{theorem}[definition]{Theorem}
\newtheorem{proposition}[definition]{Proposition}
\newtheorem{corollary}[definition]{Corollary}
\newtheorem{remark}[definition]{Remark}

\def\e{\varepsilon}
\newcommand{\weak}{\rightharpoonup}

\title{Asymptotic analysis of fractional Sobolev spaces on thin films in the low-integrability regime}
\author{Andrea Braides\footnote{Department of Mathematics, University of Rome Tor Vergata, via della ricerca scientifica 1, 00133 Rome, Italy} , Andrea Pinamonti\footnote{Department of Mathematics, University of Trento, via Sommarive 14, 38123 Povo, Italy} , and Margherita Solci\footnote{DADU, Universit\`a di Sassari, piazza Duomo 6, 07041 Alghero, Italy}}
\date{}

\begin{document}

\maketitle
\begin{abstract} 
We study the behaviour of fractional Sobolev spaces $H^s(\Omega_\varepsilon)$ with $s\in(0,1/2)$ defined on ``thin films'' $\Omega_\varepsilon=\omega\times (0,\varepsilon)$ in $\mathbb R^d$, and prove that they tend to the space $H^{s+\frac12}(\omega)$ as $\varepsilon\to 0$. This is made precise by using a notion of dimension-reduction convergence, with respect to which suitably scaled Gagliardo seminorms define equicoercive functionals. Asymptotic results are proved for $s\to 0^+$ and $s\to 1/2^-$.
\smallskip

{\bf MSC codes:} 49J45, 35R11, 73K10.

{\bf Keywords:} $\Gamma$-convergence, non-local functionals, fractional Sobolev spaces, dimension reduction, thin films.

\end{abstract}
\section{Introduction}
This paper complements an analogous analysis for thin films in the ``high integrability'' regime carried out in~\cite{BS25}. There, the behaviour of the (squared) Gagliardo $H^s$-seminorms 
\[
\int_{\Omega_\e\times\Omega_\e}
\frac{|u(x)-u(y)|^2}{|x-y|^{d+2s}}\,dx\,dy
\]
in thin domains $\Omega_\e=\omega\times(0,\e)\subset \mathbb R^d$ as their thickness $\e$ tends to $0$ is examined in the case $s>1/2$. In particular, it is shown that the critical scaling in order to obtain a dimensionally reduced energy (see, e.g.,~\cite{LDR} for problems set in integer Sobolev spaces) is $\e^{3-2s}/(1-s)$, highlighting the combined effect of the geometric scaling and the scale of relevant interactions (up to $\e$), and that the limit of the scaled energies is always a Dirichlet integral in the dimensionally reduced domain with a coefficient depending on $s$. In particular, a Bourgain--Brezis--Mironescu-type result~\cite{BBM} is proved by letting simultaneously $s\to 1^-$ and $\e\to 0^+$, obtaining in the limit
\[
\frac{\mathcal H^{d-1}(S^{d-1})}{2d}\int_\omega |\nabla'u|^2\,dx',
\]
where the prime denotes $(d-1)$-dimensional quantities and variables. This corresponds to the energy obtained by separation of scales, letting first $s\to 1$ and then $\e\to 0$.

We will show that in the case $s<1/2$ the behaviour is completely different. In order to prove that dimension reduction holds, we first follow an approach based on a ``slicing'' lemma in the thin direction. This approach differs from the one in~\cite{BS25}, which instead was based on a discretization-and-averaging argument that is not available for $s<1/2$. The slicing result allows us to estimate the integral, over $\omega$, of the one-dimensional $H^s$-seminorm squared in the $d$-th direction by the Gagliardo seminorm squared multiplied by $\e^{2s-1}$.

The second ingredient in our analysis is the determination of the relevant scaling for a dimensionally reduced limit. This turns out to be the scale of interactions at a finite distance, so that the contribution of $\e$ to the scale of the energy is due to the integral in the vertical direction (twice). This argument allows us to conclude that, setting
\[
F^s_\e(u)= \frac1{\e^2}
\int_{\Omega_\e\times\Omega_\e}
\frac{|u(x)-u(y)|^2}{|x-y|^{d+2s}}\,dx\,dy,
\]
if a sequence of functions $(u_\e)$ satisfies $F^s_\e(u_\e)\le S<+\infty$, then, upon extraction of a subsequence and possibly adding suitable translations, there exists $u\in L^2(\omega)$ such that the functions $v_\e$ defined on $\omega\times(0,1)$ by
\[
v_\e(x',x_d)=u_\e(x',\e x_d)
\]
converge weakly in $L^2(\omega\times(0,1))$ to the function $v$ given by $v(x',x_d)=u(x')$ (we refer to this as {\em dimension-reduction convergence}). Once this compactness result is proved, the $\Gamma$-limit as $\e\to 0$ can be computed with respect to the dimension-reduction convergence, showing that actually $u\in H^{\frac12+s}(\omega)$, and that 
\[
\Gamma\text{-}\lim_{\e\to 0} F^s_\e(u)= \int_{\omega\times\omega}
\frac{|u(x')-u(y')|^2}{|x'-y'|^{d+2s}}\,dx'\,dy',
\]
which is nothing but the squared $H^{\frac12+s}(\omega)$-seminorm, since 
trivially $d+2s= d-1+2\bigl(\tfrac12+s\bigr)$.

We also analyze the case $s=s_\e$ converging to $0$, and show that in this case as well
\[
\Gamma\text{-}\lim_{\e\to 0} F^{s_\e}_\e(u)= \int_{\omega\times\omega}
\frac{|u(x')-u(y')|^2}{|x'-y'|^{d}}\,dx'\,dy',
\]
that is, the squared $H^{\frac12}(\omega)$-seminorm. We note that this holds also for unbounded $\omega$, and in particular for $\omega=\mathbb R^{d-1}$, showing that a Maz'ya--Shaposhnikova-type result (which would give, in the limit, the simple $L^2$-norm squared of $u$, see~\cite{MS}) does not hold. The reason for this difference lies in the fact that the Maz'ya--Shaposhnikova result is a consequence of the non-integrability at infinity of the kernel $|\xi|^{-d-2s}$ as $s\to 0$ in $\mathbb R^d$. This lack of integrability is not an issue in the thin-films case, for which the relevant directions at infinity are the $\xi'$ coordinates in $\mathbb R^{d-1}$, for which the limit kernel $|\xi'|^{-d}$ is integrable at infinity. 

We proceed in our analysis by describing the behaviour at the first scaling at which the compactness lemma no longer ensures that the limit depends only on the variable $x'$; namely, we consider the energies 
\[
\widetilde F^s_\e(u)= \frac1{\e^{1-2s}}
\int_{\Omega_\e\times\Omega_\e}
\frac{|u(x)-u(y)|^2}{|x-y|^{d+2s}}\,dx\,dy.
\]
Since $\widetilde F^s_\e(u)= \e^{1+2s} F^s_\e(u)$, the $\Gamma$-limit is $0$ on all functions in $L^2(\omega\times (0,1))$ depending only on the $x'$-variable, but it is non-trivial for functions with a genuine dependence on the $x_d$-variable. More precisely, we have
\[
\Gamma\text{-}\lim_{\e\to 0} \widetilde F^{s}_\e(v)= C_{s,d}\int_{\omega}\int_0^1\int_0^1\frac{|v(x', x_d)-v(x', y_d)|^2}{|x_d-y_d|^{1+2s}}\, dx_d\, dy_d\, dx',
\]
where the $\Gamma$-limit is computed with respect to the weak convergence in $L^2(\omega\times(0,1))$ of the functions $v_\e(x',x_d)=u_\e(x',\e x_d)$, and
\[
C_{s,d}=\int_{\mathbb R^{d-1}}
\frac{1}{(1+|\xi'|^2)^{\frac{d}{2}+s}}\, d\xi' 
\]
is a {\em hypergeometric integral}. Finally, we show that the asymptotic scaling of the squared $H^{s_\e}(\Omega_\e)$-seminorm of characteristic functions is
$
\frac{\e^{2-2s_\e}}{s_\e}\,.
$

\section{Notation}
We use standard notation for fractional Sobolev spaces (see \cite{leofrac}). In particular, if $A$ is a measurable subset of $\mathbb R^d$ and $s\in(0,1)$, the (squared) $s$-seminorm in $H^s(A)$ is defined as  
\[
\lfloor u\rfloor_s^2(A)=\int_{A}\int_{A} \frac{|u(x)-u(y)|^2}{|x-y|^{d+2s}}\, dx\, dy.
\]

In what follows, $\omega$ will be a bounded Lipschitz subset of $\mathbb R^{d-1}$. For $\varepsilon>0$ we consider the {\em thin film} $\Omega_\varepsilon=\omega\times(0,\varepsilon)\subset\mathbb R^d$. Our aim is to analyse the behaviour of the fractional Sobolev spaces $H^s(\Omega_\varepsilon)$ as $\varepsilon\to 0$, and possibly as $s=s_\varepsilon\to s_0<\frac12$.

We remark that the assumption that $\omega$ is bounded will not be needed for most of our results.


\section{Dimension-reduction compactness}
In this section we prove a preliminary compactness result, which allows us to highlight the relevant scaling regime of the Gagliardo $H^s$-seminorm on $\Omega_\e$. The following lemma partially improves a slicing result in \cite[Lemma 6.36]{leofrac} quantifying the dependence of the scaling. Note that the hypothesis of $\omega$ bounded is not necessary.

\begin{lemma}[slicing on the ``thin'' direction]\label{slixd} There exists $C>0$ such that for all $\e>0$, $s\in(0,1)$, and $u\in H^s(\Omega_\e)$ we have
\begin{equation}\label{slixd-eq}
\int_\omega\int_0^1\int_0^1 \frac{|v(x^\prime,x_d)-v(x^\prime,y_d)|^2}{|x_d-y_d|^{1+2s}}\, dx_d\, dy_d\, dx^\prime\leq C\e^{2s-1}\int_{\Omega_\e}\int_{\Omega_\e} \frac{|u(x)-u(y)|^2}{|x-y|^{d+2s}}\, dx\, dy, 
\end{equation}
where $v\colon\omega\times (0,1)\to \mathbb R$ is the scaled function given by 
$v(x^\prime, t)=u(x^\prime,\e t)$. 
\end{lemma}
\begin{proof}  
Given $x^\prime\in\omega$ and $x_d\neq y_d\in(0,\e)$, we consider the ball $B=B(x^\prime,x_d,y_d)$ centered at $z^\ast=(x^\prime,\frac{x_d+y_d}{2})$ with radius $\frac{|x_d-y_d|}{4}$. 
Note that, for $\e$ small enough (with respect to the dimensions of $\omega$),  
$|B\cap\Omega_\e|\geq c_\omega|B|$ with $c_\omega>0$ by the Lipschitz assumption on $\omega$. Hence, 
$$\frac{|u(x^\prime,x_d)-u(x^\prime,y_d)|^2}{|x_d-y_d|^{1+2s}}\leq \frac{2}{c_\omega|B|}\int_{B\cap\Omega_\e} \frac{|u(x^\prime,x_d)-u(z)|^2+|u(z)-u(x^\prime,y_d)|^2}{|x_d-y_d|^{1+2s}}\, dz.$$ 
By Fubini's Theorem and by exchanging the role of $x_d$ and $y_d$ in second term of the sum, we then get 
\begin{eqnarray}\label{sli-part}
&&\hspace{-1cm}\int_\omega\int_0^\e\int_0^\e\frac{|u(x^\prime,x_d)-u(x^\prime,y_d)|^2}{|x_d-y_d|^{1+2s}}\, dx_d\, dy_d\, dx^\prime\nonumber\\
&&\leq 
C\int_{\Omega_\e}\int_{\Omega_\e}|u(x)-u(z)|^2 \int_0^\e\chi_{B\cap\Omega_\e}(z) \frac{1}{|x_d-y_d|^{d+1+2s}}\, dy_d\, dx\, dz,
\end{eqnarray}
with $C>0$ only depending on $c_\omega$ and the dimension $d$. 
If $\chi_{B\cap\Omega_\e}(z)=1$, then $|z-z^\ast|\leq \frac{|x_d-y_d|}{4}$, and we can estimate 
$$|x-z|\leq |z-z^\ast|+|x-z^\ast|\leq \frac{|x_d-y_d|}{4} +\frac{|x_d-y_d|}{2};$$ 
We obtain 
\begin{eqnarray*}
\int_0^\e\chi_{B\cap\Omega_\e}(z) \frac{1}{|x_d-y_d|^{d+1+2s}}\, dy_d \leq 2\int_{\frac{4|x-z|}{3}}^{+\infty} t^{-d-1-2s}\, dt=\frac{2}{d+2s}\Big(\frac{3}{4}\Big)^{d+2s} \frac{1}{|x-z|^{d+2s}}.
\end{eqnarray*}
From \eqref{sli-part} it follows that 
\begin{equation}
\int_\omega\int_0^\e\int_0^\e\frac{|u(x^\prime,x_d)-u(x^\prime,y_d)|^2}{|x_d-y_d|^{1+2s}}\, dx_d\, dy_d\, dx^\prime\leq C\int_{\Omega_\e}\int_{\Omega_\e}\frac{|u(x)-u(z)|^2}{|x-z|^{d+2s}}  dx\, dz, 
\end{equation} 
where $C>0$ does not depend on $\e$ and $s$. 
Since 
$$\int_\omega\int_0^\e\int_0^\e\frac{|u(x^\prime,x_d)-u(x^\prime,y_d)|^2}{|x_d-y_d|^{1+2s}}\, dx_d\, dy_d\, dx^\prime=
\e^{1-2s}\int_\omega\int_0^1\int_0^1\frac{|v(x^\prime,t)-v(x^\prime,\tau)|^2}{|t-\tau|^{1+2s}}\, dt\, d\tau\, dx^\prime,
$$
the claim follows.
\end{proof}

\begin{remark}[rigidity]\label{rigidity-rm}\rm 
Lemma \ref{slixd} implies a rigidity result for ``pointwise'' convergence.  
More precisely, given $u\colon\omega\times (0,1)\to\mathbb R$ and defining 
$u_\e(x^\prime,x_d)=u(x^\prime, \frac{x_d}{\e})$, 
by \eqref{slixd-eq} we obtain that 
$$\int_\omega\int_0^1\int_0^1 \frac{|u(x^\prime,x_d)-u(x^\prime,y_d)|^2}{|x_d-y_d|^{1+2s}}\, dx_d\, dy_d\, dx^\prime\leq C\e^{2s-1}\lfloor u_\e\rfloor_{s_\e}^2(\Omega_\e).$$ 
Then, if $\e^{2s-1}\lfloor u_\e\rfloor_{s_\e}^2(\Omega_\e)=o(1)_{\e\to 0}$, the function $u$ does not depend on the $d$-th variable; that is, $u(x^\prime,x_d)=v(x^\prime)$. 
\end{remark}

Lemma \ref{slixd} allows us to prove a compactness result, in the sense of the dimensional-reduction convergence defined as follows. 
\begin{definition}[(weak) dimension-reduction convergence]\label{def-conv}
Let $\{u_\e\}$ be a sequence in $H^{s_\e}(\Omega_\e)$ and let $v_\e\colon\omega\times(0,1)\to\mathbb R$ be defined by $v_\e(x)=u_\e(x^\prime,\e x_d)$. We say 
that $u_\e\to u$ with $u\in L^2(\omega)$ if 
$v_\e\weak v$ in $L^2(\omega\times(0,1))$ and $v(x)=u(x^\prime)$.  
\end{definition}

We start by showing a compactness result for sequences bounded in $L^2$. 

\begin{lemma}[dimensional-reduction compactness -- i]\label{compactness1}  
Let $\{u_\e\}$ be a bounded sequence in $L^2(\omega\times(0,1))$ such that 
$$\lim_{\e\to 0}\frac{1}{\e^{2s_\e-1}} \lfloor u_\e\rfloor_{s_\e}^2(\Omega_\e)=0.$$ 
Then
there exists $u\in L^2(\omega)$ such that, up to the addition of constants and up to subsequences, 
$u_\e\to u$ in the sense of Definition {\rm \ref{def-conv}.} 
\end{lemma}

\begin{proof}
Up to the addition of constants and up to subsequences, we have the weak convergence $v_\e\weak v$. 
To show that $v$ does not depend on the ``thin'' variable, we apply the Poincar\'e inequality and Lemma \ref{slixd}, obtaining 
\begin{eqnarray}\label{rigidity}
\int_\omega\int_0^1 |v_\e(x^\prime, x_d)- \overline v_\e(x^\prime)|^2\, dx^\prime\, dx_d&\leq&C \int_\omega \int_0^1\int_0^1 
\frac{|v_\e(x^\prime, x_d)-v_\e(x^\prime, y_d)|^2}{|x_d-y_d|^{1+2s_\e}}\, dx_d\, dy_d\, dx^\prime\nonumber \\
&\leq&C \e^{2s_\e-1} \lfloor u_\e\rfloor_{s_\e}^2(\Omega_\e)\ \leq \ C \e^{2s_\e+1}=o(1)_\e,
\end{eqnarray}
where $\overline v_\e(x^\prime)=\int_0^1v_\e(x^\prime,t)\, dt$. Since $\overline v_\e\weak \overline v$ in $L^2(\omega)$, 
where $\overline v(x^\prime)=\int_{0}^1v(x^\prime, t)\, dt$, by \eqref{rigidity} and the lower semicontinuity of the norm we get that the weak limit $v$ does not depend on $x_d$. 
\end{proof}

\begin{lemma}[dimensional-reduction compactness -- ii]\label{compactness2} Let $\{u_\e\}$ be such that 
$$\sup_\e \frac{1}{\e^2} \lfloor u_\e\rfloor_{s_\e}^2(\Omega_\e)<+\infty.$$ 
Then
there exists $u\in L^2(\omega)$ such that, up to addition of constants and up to subsequences, 
$u_\e\to u$ in the sense of Definition \rm \ref{def-conv}.  
\end{lemma}

%
%
%

\begin{proof}
By the Poincar\'e inequality (see \cite[Theorem 6.33]{leofrac}, we have that 
$$\|u_\e - \overline u_\e \|^2_{L^2(\Omega_\e)}\leq \frac{C}{\e}\lfloor u_\e\rfloor_{s_\e}^2(\Omega_\e), $$
where $\overline u_\e=\frac{1}{|\Omega_\e|}\int_{\Omega_\e} u_\e(x)\, dx$.  
Then, 
$$\|v_\e - \overline u_\e \|^2_{L^2(\omega\times(0,1))}\leq \frac{C}{\e^2}\lfloor u_\e\rfloor_{s_\e}^2(\Omega_\e),$$
which is bounded by assumption. The conclusion follows from Lemma \ref{compactness1}
\end{proof}

\section{Analysis at scale $\e^2$}
Lemma \ref{compactness2} suggests the scaling for a dimension reduction limit of the Gagliardo seminorms. This is confirmed by the following theorem.

\begin{theorem}\label{drconv}
Let $s_\e\to s_0<\frac12$; then the $\Gamma$-limit with respect to the weak dimension-reduction convergence in Definition {\rm\ref{def-conv}} of 
$$
F_\e^s(u)=\frac1{\e^2} \lfloor u\rfloor^2_{s_\e}(\Omega_\e)
$$
is given on $L^2(\omega)$ by the $H^{\frac12+s_0}$-seminorm squared on $\omega$; that is, noting that
$d-1+2(\frac12+s_0)=d+2s_0$, by 
$$
F^s_0(u)=\int_\omega\int_\omega 
\frac{|u(x')-u(y')|^2}{|x'-y'|^{d+2s_0}}dx' dy'.
$$
\end{theorem}

\begin{proof}
{\bf Lower bound.} 
By compactness the limit is independent of $x_d$.
If $u_\e\to u$; that is, $v_\e\weak u$ in $L^2(\omega\times(0,1))$, then for all fixed $\delta>0$ and $s=s_\e\to 0$
\begin{eqnarray*}\liminf_{\e\to 0}\frac{1}{\e^2} \lfloor u_\e\rfloor_{s_\e}^2(\Omega_\e)
&\ge& \liminf_{\e\to 0}\frac{1}{\e^2}\int_{(\Omega_\e\times \Omega_\e)\cap \{|x'-y'|>\delta\}}\frac{|u_\e(x)-u_\e(y)|^2}{|x-y|^{d+2s_\e}}dx dy
\\
&=&\liminf_{\e\to 0}\int_0^1\int_0^1\int_{(\omega\times\omega)\cap \{|x'-y'|>\delta\}}\frac{|v_\e(x)-v_\e(y)|^2}{|x'-y'|^{d+2s_\e}}dx dy
\\
&\ge&\int_0^1\int_0^1\int_{(\omega\times\omega)\cap \{|x'-y'|>\delta\}}\frac{|u(x')-u(y')|^2}{|x'-y'|^{d+2s_0}}dx dy
\end{eqnarray*}
The lower bound is optimized by letting $\delta\to0$.

\bigskip
{\bf Upper bound.} For $u$ Lipschitz, note that
\begin{eqnarray*}
&&\frac{1}{\e^2}\int_{(\Omega_\e\times \Omega_\e)\cap \{|x'-y'|<\delta\}}\frac{|u(x)-u(y)|^2}{|x-y|^{d+2s_\e}}dx dy
\\
&&\le \frac{1}{\e^2}\Big(\int_{(\Omega_\e\times \Omega_\e)\cap \{|x-y|< 2\e\}}\frac{|u(x)-u(y)|^2}{|x-y|^{d+2s_\e}}dx dy
\\&&\qquad +
C\int_{(\Omega_\e\times \Omega_\e)\cap \{2\e<|x'-y'|<\delta\}}\frac{|u(x)-u(y)|^{2}}{|x'-y'|^{d+2s_\e}}dx dy\Big)
\\
&&\le \frac{C}{\e^2}\Big(\int_{(\Omega_\e\times \Omega_\e)\cap \{|x-y|< 2\e\}}\frac{1}{|x-y|^{d-2+2s_\e}}dx dy
\\&&\qquad +
\e^2\int_{(\omega\times\omega)\cap \{2\e<|x'-y'|<\delta\}}\frac{1}{|x'-y'|^{d-2+2s}}dx dy\Big)
\\
&&\le \frac{C}{\e^2}\Big(\e\int_0^{2\e}t^{1-2s_\e}dt
+
\e^2\int_{\e}^{\delta}t^{-2s_\e}dt\Big)
\\
&&\le C\big(\e^{1-2s_\e}
+
\delta^{1-2s_\e}\big),
\end{eqnarray*}
while
\begin{eqnarray*}
&&\lim_{\e\to 0}\frac{1}{\e^2}\int_{(\Omega_\e\times \Omega_\e)\cap \{|x'-y'|>\delta\}}\frac{|u(x')-u(y')|^2}{|x-y|^{d+2s_\e}}dx dy
\\ &&\qquad\qquad
=\int_{(\omega\times\omega)\cap \{|x'-y'|>\delta\}}\frac{|u(x')-u(y')|^2}{|x'-y'|^{d+2s_0}}dx dy.
\end{eqnarray*}
Hence, 
$$\limsup_{\e\to 0}\frac{1}{\e^2} \lfloor u\rfloor_{s_\e}^2(\Omega_\e)
\le\int_{\omega\times\omega}\frac{|u(x')-u(y')|^2}{|x'-y'|^{d+2s_0}}dx dy,
$$
by the arbitrariness of $\delta$.

For $u\in H^{\frac12+s_0}(\omega)$ the result is obtained by approximations.
\end{proof}

\begin{corollary} 
Let $\lambda(s_\e, \e)>\!> \e^2$. Then 
$$\Gamma\hbox{\rm-}\lim_{\e\to 0}\frac{1}{\lambda(s_\e,\e)}\lfloor u\rfloor_{s_\e}^2(\Omega_\e)=0$$ 
for all $u\in L^{2}(\omega)$, where the $\Gamma$-limit is taken with respect to the convergence in Definition {\rm\ref{def-conv}}. Furthermore, 
if $u$ is Lipschitz, then also the pointwise limit is $0$. 
\end{corollary}

\section{Dimensional-reduction regimes} 
As a consequence of the previous sections, we can analyze the behaviour of energies of the form $$\frac{1}{\lambda(s_\e,\e)}\lfloor u\rfloor_{s_\e}^2(\Omega_\e)$$ 
as $\e\to 0$, $s_\e\to s_0\in [0,\frac{1}{2})$, and $$\lambda(s_\e,\e)<\!<\e^{1-2s}.$$
Thanks to Lemma \ref{compactness1}, in all these regimes we have dimensional reduction; that is, the domain of the $\Gamma$-limit is a subset of $L^2(\omega)$. 
In particular, from Theorem \ref{drconv} we obtain 
\begin{enumerate}
\item[(i)] if $\lambda<\!<\e^2$, then 
$$\Gamma\hbox{\rm-}\lim_{\e\to 0}\frac{1}{\lambda(s_\e,\e)}\lfloor u\rfloor_{s_\e}^2(\Omega_\e)=F_0(u)=\begin{cases} 
0 & \hbox{\rm if }\ u \ \hbox{\rm is constant}\\
+\infty & \hbox{\rm otherwise; }
\end{cases}$$   
\item[(ii)] ($\lambda\sim\e^2$) 
$$\Gamma\hbox{\rm-}\lim_{\e\to 0}\frac{1}{\e^2}\lfloor u\rfloor_{s_\e}^2(\Omega_\e)=\int_\omega\int_\omega 
\frac{|u(x')-u(y')|^2}{|x'-y'|^{d+2s_0}}\, dx^\prime \, dy^\prime=\lfloor u\rfloor_{s_0+\frac12}^2(\omega);$$ 
\item[(iii)] if $\e^2<\!<\lambda<\!<\e^{1-2s_\e}$, then 
$$\Gamma\hbox{\rm-}\lim_{\e\to 0}\frac{1}{\lambda(s_\e,\e)}\lfloor u\rfloor_{s_\e}^2(\Omega_\e)=0.
$$  
\end{enumerate}

Note that by the Brezis--Bourgain--Mironescu result applied in $\omega$ we have 
$$
\Gamma\hbox{\rm-}\lim_{s\to \frac12^-}
\big(\frac12-s\big)\lfloor u\rfloor_{s+\frac12}^2(\omega)=
\frac{\mathcal H^{d-2}(S^{d-2})}{2(d-1)}\int_\omega |\nabla'u|^2dx',
$$
which gives an approximate description of case (ii) as $s_0\to\frac12^-$, and a link with the results in \cite{BS25}.

\section{Limit energies defined on full-dimensional sets}
For $\tau>0$ small enough, let $\omega_\tau=\{x^\prime\in\omega: \ \hbox{\rm dist}(x^\prime,\partial\omega)>\tau\}$ and $\Omega_\e^\tau=\omega_\tau\times (0,\e)$. 

\begin{lemma}[estimate of the seminorm in the thin direction]\label{vert-lemma}
Let $u_\e\colon\Omega_\e\to \mathbb R$ be such that there exists $L>0$ such that for all $\e>0$ 
\begin{equation}\label{equilip}
|u_\e(x^\prime, x_d)-u_\e(y^\prime, x_d)|\leq L|x^\prime-y^\prime|
\end{equation}
for all $x^\prime,y^\prime\in\omega$ and $x_d\in (0,\e)$. 
Then, with fixed $\eta>0$ and $\tau>0$,  
\begin{eqnarray}\label{vert-est1}
&&(1+\eta)\lfloor u_\e\rfloor_{s_\e}^2(\Omega_\e)+\frac{\e\phi(\e)}{\eta}\geq C_{s_\e,d} \int_{\omega_\tau}\int_0^\e\int_0^\e\frac{|u_\e(x^\prime, x_d)-u_\e(x^\prime, y_d)|^2}{|x_d-y_d|^{1+2s_\e}}\, dx_d\, dy_d\, dx^\prime\qquad \\
&&\lfloor u_\e\rfloor_{s_\e}^2(\Omega^\tau_\e)-\frac{\e\phi(\e)}{\eta}\leq (1+\eta)C_{s_\e,d} \int_{\omega_\tau}\int_0^\e\int_0^\e\frac{|u_\e(x^\prime, x_d)-u_\e(x^\prime, y_d)|^2}{|x_d-y_d|^{1+2s_\e}}\, dx_d\, dy_d\, dx^\prime,\qquad \label{vert-est2}
\end{eqnarray}
where 
\begin{equation}\label{defC}
C_{s,d}=\int_{\mathbb R^{d-1}}
\frac{1}{(1+|\xi|^2)^{\frac{d}{2}+s}}\, d\xi 
\end{equation}
and $\phi$ only depends on $L$, $\omega$, $\tau$, and on the supremum of the $\infty$ norms of $u_\e$, and it is such that $\lim\limits_{\e\to 0}\phi(\e)=0$. 
\end{lemma}

\begin{proof} 
We extend $u_\e$ to $\mathbb R^{d-1}\times(0,\e)$ by setting $u_\e(x^\prime,x_d)=0$ for $x^\prime\in\mathbb R^{d-1}\setminus\omega$. 
Note that, by a change of variable, 
for all $x_d\neq y_d$ and $x^\prime\in\omega$ 
we have 
\begin{eqnarray}
\frac{C_{s,d}}{|x_d-y_d|^{1+2s}}&=&
\int_{\mathbb R^{d-1}} \frac{|x_d-y_d|^{d-1}}{(|x_d-y_d|^2+(|x_d-y_d||\xi|)^2)^{\frac{d}{2}+s}}\, d\xi\nonumber\\
&=&
\int_{\mathbb R^{d-1}} \frac{1}{(|x_d-y_d|^2+|x^\prime-y^\prime|^2)^{\frac{d}{2}+s}}\, dy^\prime,\label{rd1}
\end{eqnarray} 
so that, by using Fubini's Theorem, we can write  
\begin{eqnarray}
&&\hspace{-2cm}C_{s_\e,d}\int_{\omega_\tau}\int_0^\e\int_0^\e\frac{|u_\e(x^\prime, x_d)-u_\e(x^\prime, y_d)|^2}{|x_d-y_d|^{1+2s_\e}}\, cx_d\, dy_d\, dx^\prime\nonumber\\
&=&
\int_{\omega_\tau}\int_0^\e\int_0^\e\int_{\mathbb R^{d-1}} \frac{|u_\e(x^\prime, x_d)-u_\e(x^\prime, y_d)|^2}{(|x_d-y_d|^2+|x^\prime-y^\prime|^2)^{\frac{d}{2}+s_\e}}\, dy^\prime\, dx_d\, dy_d\, dx^\prime.\label{cvar}
\end{eqnarray}
Setting, for $A$ measurable subset of $\mathbb R^{d-1}$, 
\begin{eqnarray*}
&&I^{\rm v}_{\e,\tau}(A)=\int_{\omega_\tau}\int_0^\e\int_0^\e\int_A \frac{|u_\e(x^\prime, x_d)-u_\e(x^\prime, y_d)|^2}{(|x_d-y_d|^2+|x^\prime-y^\prime|^2)^{\frac{d}{2}+s_\e}}\, dy^\prime\, dx_d\, dy_d\, dx^\prime\\ 
&&I^{\rm h}_{\e,\tau}(A)=\int_{\omega_\tau}\int_0^\e\int_0^\e\int_A \frac{|u_\e(y^\prime, y_d)-u_\e(x^\prime, y_d)|^2}{(|x_d-y_d|^2+|x^\prime-y^\prime|^2)^{\frac{d}{2}+s_\e}}\, dy^\prime\, dx_d\, dy_d\, dx^\prime  \\
&&I_{\e,\tau}(A)= \int_{\omega_\tau}\int_0^\e\int_0^\e\int_A\frac{|u_\e(x^\prime, x_d)-u_\e(y^\prime, y_d)|^2}{(|x_d-y_d|^2+|x^\prime-y^\prime|^2)^{\frac{d}{2}+s_\e}}\, dy^\prime\, dx_d\, dy_d\, dx^\prime, 
\end{eqnarray*}
a triangular inequality gives 
\begin{eqnarray}
&&I_{\e,\tau}^{\rm v}(\mathbb R^{d-1})\leq (1+\eta)I_{\e,\tau}(\mathbb R^{d-1})+\frac{1}{\eta}I_{\e,\tau}^{\rm h}(\mathbb R^{d-1})\label{tr1}\\
&&I_{\e,\tau}(\mathbb R^{d-1})\leq (1+\eta)I_{\e,\tau}^{\rm v}(\mathbb R^{d-1})+\frac{1}{\eta}I_{\e,\tau}^{\rm h}(\mathbb R^{d-1})\label{tr2}
\end{eqnarray}
for all $\eta>0$. 
We now show that 
\begin{equation}\label{stima-or}
I_{\e,\tau}^{\rm h}(\mathbb R^{d-1})\leq\e\phi(\e),
\end{equation}
where $\phi$ depends on $L$, $\omega$ and $\tau$, and $\lim\limits_{\e\to 0}\phi(\e)=0$.
Noting that for all $x^\prime\in\omega_\tau$ 
\begin{equation*}
\int_{\mathbb R^{d-1}\setminus\omega} \frac{1}{(|x_d-y_d|^2+|x^\prime-y^\prime|^2)^{\frac{d}{2}+s_\e}} dy^\prime
\leq\int_{\mathbb R^{d-1}\setminus B_\tau(0)} \frac{1}{|\xi^\prime|^{d+2s_\e}}\, d\xi^\prime
=\frac{\sigma_{d-1}}{(1+2s_\e)\tau^{1+2s_\e}}, 
\end{equation*}
where $\sigma_{d-1}=\mathcal H^{d-2}(\partial B_1(0))$, 
we get 
\begin{equation}
I^{\rm h}_{\e,\tau}(\mathbb R^{d-1}\setminus \omega)
\leq \frac{\sigma_{d-1}\|u_\e\|^2_\infty |\omega_\tau|}{\tau^{1+2s_\e}}\e^2\label{st1}\\
\end{equation}
Hence, we only have to estimate the integral in the set where $y^\prime\in\omega$. 
We have 
\begin{eqnarray*}
I^{\rm h}_{\e,\tau}(\omega)
&\leq& 
2L^2\int_{\omega_\tau}\int_0^\e\int_{\omega}|x^\prime-y^\prime|^2\int_0^\e\frac{1}{(t^2+|x^\prime-y^\prime|^2)^{\frac{d}{2}+s}}\, dt\, dy^\prime\, dy_d\, dx^\prime\\
&\leq& 
2L^2\e\phi(\e)
\end{eqnarray*}
where 
\begin{equation}\label{defphi}\phi(\e)=
\int_{\omega}\int_{\omega}|x^\prime-y^\prime|^2
\int_0^\e\frac{1}{(t^2+|x^\prime-y^\prime|^2)^{\frac{d}{2}+s}}\, dt
\, dy^\prime\, dx^\prime.
\end{equation}
Since the integrand in \eqref{defphi} pointwise converges to $0$ almost everywhere, and it is estimated by 
\begin{eqnarray*}
|x^\prime-y^\prime|^2\int_0^\e\frac{1}{(t^2+|x^\prime-y^\prime|^2)^{\frac{d}{2}+s}}\, dt
&\leq&|x^\prime-y^\prime|^{3-d-2s_\e} \int_0^{+\infty}\frac{1}{(1+w^2)^{\frac{d}{2}+s_\e}} d w \\
&\leq&\hbox{\rm diam}(\omega)^{1-2s_\e} |x^\prime-y^\prime|^{2-d}\int_0^{+\infty}\frac{1}{(1+w^2)^{\frac{d}{2}+s_\e}} d w, 
\end{eqnarray*}
with
$$\int_{\omega}\int_{\omega}|x^\prime-y^\prime|^{2-d}\, dx^\prime\, dy^\prime<+\infty,$$ 
by Lebesgue's Theorem we obtain 
$$\lim_{\e\to 0}\phi(\e)=0.$$
Since \eqref{st1} holds, this implies \eqref{stima-or}.

Noting that \eqref{st1} holds also for $I_{\e,\tau}(\mathbb R^{d-1}\setminus\omega)$, we can write
\begin{equation}\label{stima-semi}
I_{\e,\tau}(\mathbb R^{d-1})
= I_{\e,\tau}(\omega)+r_\e, 
\end{equation}
with $|r_\e|\leq \sigma_{d-1}\|u_\e\|^2_\infty |\omega_\tau|\tau^{-1-2s_\e}\e^2$. 

Recalling \eqref{tr1} and \eqref{tr2}, estimates \eqref{stima-or} and \eqref{stima-semi} imply that 
\begin{eqnarray*}
I_{\e,\tau}^{\rm v}(\mathbb R^{d-1})&\leq&(1+\eta)(I_{\e,\tau}(\omega)+|r_\e|)+\frac{\e\phi(\e)}{\eta}\\
&\leq&(1+\eta)\lfloor u_\e\rfloor_{s_\e}^2(\Omega_\e)
+C\e^2\tau^{-1-2s_\e}
+\frac{\e\phi(\e)}{\eta}
\end{eqnarray*}
\begin{eqnarray*}
\lfloor u_\e\rfloor_{s_\e}^2(\Omega_\e^\tau)-C\e^2\tau^{-1-2s_\e}&\leq& I_{\e,\tau}(\omega)-|r_\e|\\
&\leq&(1+\eta)I_{\e,\tau}^{\rm v}(\mathbb R^{d-1})+\frac{\e\phi(\e)}{\eta}, 
\end{eqnarray*}
and since \eqref{cvar} holds the claim follows. \end{proof}



\begin{definition}\label{def2}
Let $u_\e\colon\Omega_\e\to\mathbb R$ and $v\colon\omega\times(0,1)\to\mathbb R$. Setting $v_\e(x^\prime, x_d)=u_\e(x^\prime, \e x_d)$, we say that $u_\e\to v$ if 
$v_\e\to v$ in $L^2_{\rm loc}(\omega\times(0,1))$. 
\end{definition}

\begin{theorem} 
Let $s_\e\to s_0\in[0,\frac{1}{2})$. 
The sequence of functionals 
$$E_\e(u)=\frac{1}{\e^{1-2s_\e}}\lfloor u\rfloor_{s_\e}^2(\Omega_\e)$$ 
$\Gamma$-converges with respect to the convergence in Definition {\rm\ref{def2}} to the functional 
$$E(v)=C_{s_0,d}\int_{\omega}\int_0^1\int_0^1\frac{|v(x^\prime, x_d)-v(x^\prime, y_d)|^2}{|x_d-y_d|^{1+2s_0}}\, dx_d\, dy_d\, dx^\prime.$$
\end{theorem} 

\begin{proof}
Let $u_\e\to v$ as in Definition \ref{def2}. 
Note that we can suppose that the sequence $\{u_\e\}$ is equibounded in $L^\infty(\Omega_\e)$ up to a truncation argument, since the functionals $E_\e$ decrease by truncation. 
Let $\tau>0$ and $\varphi_\tau\colon\mathbb R^{d-1}\to [0,+\infty)$ be a mollifier with support in $B_\tau(0)$. 
We set $u_\e^\tau(x^\prime,x_d)=u_\e \ast\varphi_\tau$, where the convolution is performed in the variable $x^\prime$. Noting that 
$$\lfloor u_\e\rfloor_{s_\e}^2(\Omega_\e)\geq \lfloor u_\e^\tau\rfloor_{s_\e}^2(\omega_\tau\times(0,\e)),$$ 
we can apply Lemma \ref{vert-lemma} to the sequence $\{u_\e^\tau\}$ with $\Omega_\e$ replaced by $\Omega_\e^\tau=\omega_{\tau}\times (0,\e)$. Then, by \eqref{vert-est1} and the arbitrariness of $\eta>0$ therein, noting that $\e^{2s_\e}\phi(\e)\to 0$ we get 
\begin{eqnarray*}
\liminf_{\e\to 0} E_\e(u_\e)&\geq&C_{s_0,d} \liminf_{\e\to 0}\e^{2s_\e-1}\int_{\omega_{2\tau}}\int_0^\e\int_0^\e\frac{|u^\tau_\e(x^\prime, x_d)-u^\tau_\e(x^\prime, y_d)|^2}{|x_d-y_d|^{1+2s_\e}}\, dx_d\, dy_d\, dx^\prime\\
&=&C_{s_0,d} \liminf_{\e\to 0}\int_{\omega_{2\tau}}\int_0^1\int_0^1\frac{|v^\tau_\e(x^\prime, x_d)-v^\tau_\e(x^\prime, y_d)|^2}{|x_d-y_d|^{1+2s_\e}}\, dx_d\, dy_d\, dx^\prime, 
\end{eqnarray*}
where $v^\tau_\e(x^\prime, x_d)=u^\tau_\e(x^\prime, \e x_d)$. 
If we also set $v^\tau=v\ast \varphi_\tau$, we have the convergence $v^\tau_\e\to v^\tau$ in $L^2(\omega_\tau\times (0,1))$. Hence, by Fatou's Lemma, 
\begin{eqnarray*}
&&\hspace{-2cm}\liminf_{\e\to 0}\int_{\omega_{2\tau}}\int_0^1\int_0^1\frac{|v^\tau_\e(x^\prime, x_d)-v^\tau_\e(x^\prime, y_d)|^2}{|x_d-y_d|^{1+2s_\e}}\, dx_d\, dy_d\, dx^\prime\\
&&\geq 
\int_{\omega_{2\tau}}\int_0^1\int_0^1\frac{|v^\tau(x^\prime, x_d)-v^\tau(x^\prime, y_d)|^2}{|x_d-y_d|^{1+2s_0}}\, dx_d\, dy_d\, dx^\prime. 
\end{eqnarray*}
Taking the limit as $\tau\to 0$ we obtain the lower bound. 

The upper bound for a  function $v\colon\mathbb R^{d-1}\times(0,1)\to\mathbb R$ such that  
\begin{equation*}
|v(x^\prime, x_d)-v(y^\prime, x_d)|\leq L|x^\prime-y^\prime|
\end{equation*}
for all $x^\prime,y^\prime\in \mathbb R^{d-1}$ and $x_d\in (0,1)$, is achieved by the trivial recovery sequence $u_\e(x^\prime, x_d)=v(x^\prime, \frac{x_d}{\e})$, and using \eqref{vert-est2}. For a general $v\in L^2(\omega\times (0,1))$ we can proceed by approximation.
\end{proof}

\section{Analysis at the critical scale for pointwise convergence} 
We have seen that by scaling by $\e^2$ the $\Gamma$-limit of (squared) seminorms in $H^s$ gives (squared) seminorms in $H^{s+\frac12}$, with a gain of regularity. In particular some characteristic functions belong to $H^2(\Omega_\e)$, but not to the limit fractional Sobolev space. 
In this section, we will examine the pointwise behaviour of a suitable scaling of $\lfloor u\rfloor_{s_\e}^2(\Omega_\e) $ as $\e, s_\e\to 0$ for discontinuous $u$, showing that it is not trivial at some scale $\lambda(s_\e,\e)$ at which the $\Gamma$-limit is $0$. We will restrict our analysis to the case $d=2$, the general case resulting only in technical changes (see Remark \ref{rem-d}).

It is not restrictive to assume that 
\begin{equation}\label{limes}
\lim_{\e\to 0}\e^{s_\e}=\varrho\in [0,1] 
\end{equation}
exists, up to passing to a subsequence.

\begin{lemma}[asymptotic estimate of the seminorm of piecewise-constant functions]\label{lemma-ct}  
Let $u\colon (a,b)\times (0,\e)\to \mathbb R$ be such that $u(x_1,x_2)=v(x_1)$, with 
$v$ piecewise constant.  
Then, 
\begin{equation}\label{ct-eq}
\lfloor u\rfloor_{s_\e}^2(\Omega_\e)=\frac{\e^{2-2s_\e}}{s_\e}(1-\varrho^2)\sum_{t\in S(v)} |v(t^+)-v(t^-)|^2+o\Big(\frac{\e^{2-2s_\e}}{s_\e}\Big)_{\e\to 0},
\end{equation}
where $\varrho=\lim\limits_{\e\to 0}\e^{s_\e}$ as defined in \eqref{limes}. 
\end{lemma} 
\begin{proof}
We can write the piecewise-constant function $v$ as 
$v(t)=\sum_{j=1}^{N+1} \lambda_j \chi_{(a_{j-1}, a_j)}(t)$, with 
$a_{j-1}<a_j$ for all $j\in\{1,\dots, N+1\}$, $a=a_0$, $b=a_{N+1}$ and $\lambda_j\in\mathbb R$.

For all $\tau>0$, there exists $C_\tau$ not depending on $\e$ such that we can estimate 
\begin{equation}\label{lont_tau}
\int_{\Omega_\e}\int_{\Omega_\e\cap\{|x_1-y_1|\geq\tau\}}\frac{|u(x)-u(y)|^2}{|x-y|^{2+2s_\e}}  \, dx\, dy 
\leq C_\tau\e^{2}=o\Big(\frac{\e^{2-2s_\e}}{s_\e}\Big)_{\e\to 0}.\end{equation}
Hence, in the computation of $s_\e$-seminorm in \eqref{ct-eq}, the integral over the set where $|x_1-y_1|\geq\tau$ is negligible since $s_\e \e^{2s_\e}\to 0$ as $\e\to 0$. 

By choosing $\tau$ small enough, this observation allows us to reduce the computation to the sum of contributions each one close to a jump point. 
Indeed, fixing $\tau\in\big(0,\frac{1}{2}\min\{a_j-a_{j-1}\}\big)$ and setting $I_j^{\tau,\e}=(a_j-\tau, a_j+\tau)\times(0,\e)$, we get 
\begin{equation}\label{salti_sep}
\int_{\Omega_\e}\int_{\Omega_\e\cap\{|x_1-y_1|<\tau\}}\frac{|u(x)-u(y)|^2}{|x-y|^{2+2s_\e}}  \, dx\, dy
=\sum_{j=1}^N   \int_{I_j^{\tau,\e}}\int_{I_j^{\tau,\e}}\frac{|u(x)-u(y)|^2}{|x-y|^{2+2s_\e}}  \, dx\, dy.
\end{equation} 
Each one of these integrals related to the jump point $a_j$ can be computed separately. 

Let $k\in\mathbb N$ be fixed; for all $\e>0$ such that $k\e<\tau$, 
we can write 
\begin{equation*}
\int_{I_j^{\tau,\e}}\int_{I_j^{\tau,\e}}\frac{|u(x)-u(y)|^2}{|x-y|^{2+2s_\e}}  \, dx\, dy= 
2h_j^2(I_1+I_2),  
\end{equation*}
where $h_j=v(a_j^+)-v(a_{j}^-)$, and $I_1$ and $I_2$, after a translation of $a_j$ to $0$, are 
\begin{eqnarray*}
I_1&=&\int_{(-k\e,0)\times (0,\e)}\int_{(0,k\e)\times (0,\e)}\frac{1}{|x-y|^{2+2s_\e}}  \, dx\, dy\\
I_2&=&\int_{(-\tau,0)\times (0,\e)}\int_{(k\e,\tau)\times (0,\e)}\frac{1}{|x-y|^{2+2s_\e}}  \, dx\, dy\\ 
&&+\int_{(-\tau,-k\e)\times (0,\e)}\int_{(0,k\e)\times (0,\e)}\frac{1}{|x-y|^{2+2s_\e}}  \, dx\, dy, 
\end{eqnarray*} 
independent of $j$.

Note that, if $|x_1-y_1|\geq k\e$, 
then $|x_1-y_1|^2\leq |x-y|^2\leq \frac{k^2+1}{k^2}|x_1-y_1|^2$; hence, to estimate $I_2$ both from above and below, we can consider $y_1-x_1$ in the place of $|x-y|$. 
By computing the integrals we obtain 
\begin{equation}\label{stimacost1}
\frac{k^2}{k^2+1}\frac{\e^2}{2s_\e} \phi(s_\e) \Big(\big(\frac{k\e}{\tau}\big)^{-2s_\e}-1\Big) \leq I_2\leq \frac{\e^2}{2s_\e} \phi(s_\e) \Big(\big(\frac{k\e}{\tau}\big)^{-2s_\e}-1\Big), 
\end{equation}
where $\phi(s)=\frac{(2-2^{-s})\tau^{-2s}}{1+2s}=1+O(s)_{s\to 0}$.  
As for the first integral, since we can suppose $s_\e<\frac{1}{2}$, we estimate 
\begin{eqnarray}\label{stimacost2}
I_1&\leq&\int_{(-\e,0)\times(0,\e)} \int_{B_{3\e}(x)\setminus B_{|x_1|}(x)}\frac{1}{|x-y|^{2+2s_\e}}\, dx 
=\frac{\pi\e^{2-2s_\e}}{s_\e}\big( \frac{1}{1-2s_\e}  -3^{-2s_\e}\big)\nonumber\\
&=& \frac{\e^{2-2s_\e}}{s_\e} O(s_\e)_{\e\to 0}.  
\end{eqnarray} 
We now compute the pointwise limit of $\frac{s_\e}{\e^{2-2s_\e}}\lfloor u\rfloor_{s_\e}^2(\Omega_\e)$. 
Recalling that $\phi(s_\e)- 1=O(s_\e)_{\e\to 0}$, estimates \eqref{stimacost1} and \eqref{stimacost2} imply that  
\begin{equation}\label{stimacost3}
\frac{k^2}{k^2+1}\e^{2s_\e} \Big(\big(\frac{k\e}{\tau}\big)^{-2s_\e}-1\Big) h_j^2\leq 
\frac{s_\e}{\e^{2-2s_\e}}\lfloor u\rfloor_{s_\e}^2(I_j^{\tau,\e})\leq\e^{2s_\e}\Big(\big(\frac{k\e}{\tau}\big)^{-2s_\e}-1\Big)h_j^2
\end{equation}
up to a term of order $s_\e$ ($O(s_\e)_{\e\to 0}$). 
\begin{enumerate}
\item[$\bullet$] $s_\e\log\e\to -\infty$.
In this case, $\e^{s_\e}\to 0$; by \eqref{stimacost3} it follows that 
\begin{equation}\label{stimacost4}
\frac{k^{2}}{k^2+1} h_j^2
\leq \lim_{\e\to 0}\frac{s_\e}{\e^{2-2s_\e}}\lfloor u\rfloor_{s_\e}^2(I_j^{\tau,\e})\leq 
h_j^2.
\end{equation}
Recalling \eqref{lont_tau} and \eqref{salti_sep}, and summing up over the jump set of $v$, since $k\in \mathbb N$ is arbitrary, we then obtain 
$$\lim_{\e\to 0} \frac{s_\e}{\e^{2-2s_\e}}\lfloor u\rfloor_{s_\e}^2(\Omega_\e)=\sum_{t\in S(v)} |v(t^+)-v(t^-)|^2.$$
\item[$\bullet$] $\lim\limits_{\e\to 0}s_\e\log\e\in(-\infty,0)$.
In this case, $\e^{s_\e}\to \varrho\in(0,1)$; by \eqref{stimacost3} it follows that 
\begin{equation}\label{stimacost5}
\frac{k^2}{k^2+1}(1-\varrho^2) h_j^2\leq 
\lim_{\e\to 0}\frac{s_\e}{\e^{2-2s_\e}}\lfloor u\rfloor_{s_\e}^2(I_j^{\tau,\e})\leq(1-\varrho^2)h_j^2. 
\end{equation}
Recalling \eqref{lont_tau} and \eqref{salti_sep}, and summing up over the jumps set of $v$, we then obtain 
$$\lim_{\e\to 0} \frac{s_\e}{\e^{2-2s_\e}}\lfloor u\rfloor_{s_\e}^2(\Omega_\e)=(1-\varrho^2)\sum_{t\in S(v)} |v(t^+)-v(t^-)|^2.$$
\item[$\bullet$] $\lim\limits_{\e\to 0}s_\e\log\e=0$.
In this case, $\e^{s_\e}\to 1$, and \eqref{stimacost3} implies that $\lim\limits_{\e\to 0}\frac{s_\e}{\e^{2-2s_\e}}\lfloor u\rfloor_{s_\e}^2(I_j^{\tau,\e})=0.$ 
Recalling \eqref{lont_tau} and \eqref{salti_sep}, we obtain
$$\lim_{\e\to 0} \frac{s_\e}{\e^{2-2s_\e}}\lfloor u\rfloor_{s_\e}^2(\Omega_\e)=0.$$ 
\end{enumerate} 
\end{proof}  
This asymptotic analysis suggests to consider the scaled functionals 
$$F_\e^{s}(u)=
\frac{1}{\lambda(s_\e,\e)}\lfloor u\rfloor_{s}^2(\Omega_\e),$$
where 
\begin{equation}\label{lambda}
\lambda(s,\e)=\begin{cases}
\e^{2-2s}s^{-1} & {\rm if }\ \varrho=0\\
\e^{2}s^{-1} & {\rm if }\ \varrho\in (0,1)\\
\e^2|\log\e| & {\rm if }\ \varrho=1.
\end{cases}
\end{equation}

Note that in all cases $\lambda(s,\e)>\!>\e^2$. 
In the case $\varrho=0$; that is, $s_\e=o(|\log\e|^{-1})_{\e\to 0}$, the form of $\lambda(s,\e)$ is obtained from the expansion in \eqref{stimacost3}. Since ...  
\begin{equation}\label{exp}
\Big(\frac{k\e}{\tau}\Big)^{-2s_\e}-1=e^{-2s_\e(\log \e+\log k-\log\tau)}-1=-2s_\e\ln \e
+o(s_\e\log\e)_{\e\to 0},
\end{equation}
by \eqref{stimacost3} it follows that $\frac{s_\e}{\e^{2-2s_\e}}\lfloor u\rfloor_{s_\e}^2(\Omega_\e)$ is of order $s_\e|\log \e|$; that is, $\lfloor u\rfloor_{s_\e}^2(\Omega_\e)\sim \e^{2}|\log \e|$. 
In order to treat the case $\varrho=1$, we need some restriction on the exponent $s_\e$. 
\begin{lemma}[pointwise limit in the case $\varrho=1$]\label{lemma-log}   
Let 
$$\e^2 |\log\e|^{-1} <\!<s_\e<\!< |\log\e|^{-1}$$ 
as $\e\to 0$.  
Let $u\colon (a,b)\times (0,\e)\to \mathbb R$ be such that $u(x_1,x_2)=v(x_1)$, with 
$v$ piecewise constant. 
Then,
$$\lim_{\e\to 0} \frac{1}{\e^2|\log\e|}\lfloor u \rfloor^2_{s_\e}(\Omega_\e)=2\sum_{t\in S(v)} |v(t^+)-v(t^-)|^2.$$ 
\end{lemma} 
\begin{proof}
Noting that, since $\e^2 |\log\e|^{-1} <\!<s_\e$, estimate \eqref{lont_tau} implies that also for $G^{s_\e}_\e(u)$ the integral over the set where $|x_1-y_1|\geq\tau$ is negligible, we can follow word for word the proof of Lemma \ref{lemma-ct}. 
By \eqref{stimacost3} and \eqref{exp}, we get 
\begin{equation}\label{stimaG1}
\frac{k^2}{k^2+1}2h_j^2 \leq 
\frac{1}{\e^2|\log\e|}\lfloor u \rfloor^2_{s_\e}(I_{j}^{\e,\tau})
\leq 2h_j^2, 
\end{equation}
up to a term which is infinitesimal as $\e\to 0$. 
We then conclude 
$$\lim_{\e\to 0} \frac{1}{\e^2|\log\e|}\lfloor u \rfloor^2_{s_\e}(\Omega_\e)=2\sum_{t\in S(v)} |v(t^+)-v(t^-)|^2.$$
\end{proof} 

By Lemma \ref{lemma-ct} and Lemma \ref{lemma-log} we obtain the following pointwise-convergence result. 

\begin{proposition}\label{pconv} Let $u\colon (a,b)\times (0,\e)\to \mathbb R$ be such that $u(x_1,x_2)=v(x_1)$, with $v$ piecewise constant, $\lambda(s,\e)$ be defined as in \eqref{lambda}, and $s_\e\to 0$ as $\e\to 0$.  
Then: 
\begin{enumerate}
    \item[\rm (i)] if $s_\e>\!> |\log\e|^{-1}$, then 
    $$\lim_{\e\to 0}\frac{1}{\lambda(s_\e,\e)}\lfloor u \rfloor^2_{s_\e}(\Omega_\e)=\sum_{t\in S(v)}|v(t^+)-v(t^-)|^2;$$
    \item[\rm (ii)] if $s_\e\sim |\log\e|^{-1}$, then 
 $$\lim_{\e\to 0}\frac{1}{\lambda(s_\e,\e)}\lfloor u \rfloor^2_{s_\e}(\Omega_\e)=\frac{1-\varrho^2}{\varrho^2}\sum_{t\in S(v)}|v(t^+)-v(t^-)|^2,$$
    where $\varrho=\lim\limits_{\e\to 0}\e^{s_\e}$;
    \item[\rm(iii)] if $|\log\e|^{-1}>\!> s_\e >\!> \e^2 |\log\e|^{-1}$, then 
     $$\lim_{\e\to 0}\frac{1}{\lambda(s_\e,\e)}\lfloor u \rfloor^2_{s_\e}(\Omega_\e)=2\sum_{t\in S(v)}|v(t^+)-v(t^-)|^2.$$
\end{enumerate}
\end{proposition}

\begin{remark}[$d$-dimensional estimate]\label{rem-d}\rm 
In the case $d>2$, the analog of Lemma \ref{lemma-ct} gives an estimate for the asymptotic behaviour of the $s$-seminorm.  
As a result, if $u\colon \omega\times (0,\e)\to \mathbb R$ is such that $u(x^\prime,x_d)=v(x^\prime)$, with 
$v$ piecewise constant, and $\varrho=\lim\limits_{\e\to 0}\e^{s_\e}\in [0,1]$, then 
$$\lfloor u \rfloor_{s_\e}^2(\Omega_\e)\sim \frac{\e^{2-2s_\e}}{s_\e}\int_{S(v)} |v((x^\prime)^+)-v((x^\prime)^-)|^2\, d\mathcal H^{d-2}(x^\prime),$$ 
giving a non-trivial pointwise limit as in Proposition \ref{pconv}. 
\end{remark}

\bibliographystyle{abbrv}
\bibliography{references}

\end{document}